\documentclass[11pt,reqno]{amsart}

\usepackage{amssymb,amsmath,amsthm,amsfonts}
\usepackage{bbm}
\usepackage{a4wide}

\usepackage{bookmark}
\usepackage{hyperref}
\hypersetup{pdfstartview={FitH}}

\newtheorem{theorem}{Theorem}
\newtheorem{lemma}[theorem]{Lemma}

\newtheorem{proposition}[theorem]{Proposition}

\numberwithin{equation}{section}

\begin{document}

\title[Boxes, extended boxes, and sets of positive upper density]{Boxes, extended boxes, and sets of positive upper density in the Euclidean space}

\author[P. Durcik]{Polona Durcik}
\address{Polona Durcik, California Institute of Technology, 1200 E California Blvd, Pasadena, CA 91125, USA
\& Schmid College of Science and Technology, Chapman University, One University Drive, Orange, CA 92866, USA}
\email{durcik@caltech.edu}
\email{durcik@chapman.edu}

\author[V. Kova\v{c}]{Vjekoslav Kova\v{c}}
\address{Vjekoslav Kova\v{c}, Department of Mathematics, Faculty of Science, University of Zagreb, Bijeni\v{c}ka cesta 30, 10000 Zagreb, Croatia}
\email{vjekovac@math.hr}

\date{August 27, 2020}



\begin{abstract}
We prove that sets with positive upper Banach density in sufficiently large dimensions contain congruent copies of all sufficiently large dilates of three specific higher-dimensional patterns. These patterns are: $2^n$ vertices of a fixed $n$-dimensional rectangular box, the same vertices extended with $n$ points completing three-term arithmetic progressions, and the same vertices extended with $n$ points completing three-point corners. Our results provide common generalizations of several Euclidean density theorems from the literature.
\end{abstract}

\maketitle


\section{Introduction}

Euclidean Ramsey theory typically seeks for a given pattern, such as vertices of a square, an arithmetic progression, etc., in a single partition class determined by an arbitrary (or only measurable) coloring of the Euclidean space. Stronger results than the mere coloring theorems are the so-called density theorems, which establish existence of the pattern inside an arbitrary measurable subset of positive density. The appropriate notion of density for this purpose is the \emph{upper Banach density}, defined as
\begin{equation}\label{eq:defofdensity}
\overline{\delta}(A) := \limsup_{N\to\infty} \sup_{x\in\mathbb{R}^d} \frac{|A \cap (x+[0,N]^d)|}{N^d}
\end{equation}
for any measurable $A\subseteq\mathbb{R}^d$. Here and in what follows, $|B|$ denotes the Lebesgue measure of a measurable set $B\subseteq\mathbb{R}^d$.

An interesting class of density results tries to find congruent copies of all sufficiently large dilates of a given pattern. There is always a critical dimension $d_{\textup{min}}$ below which positive statements cannot hold. Since many dimension-related issues are still unresolved, one is often content with proving the claim when $d$ is sufficiently large. An initial result of this type starts with the simplest possible pattern, a pair of points in $\mathbb{R}^d$ for $d\geq 2$, and it was established independently by Bourgain \cite{B86:roth}, Falconer and Marstrand \cite{FM86:dist}, and Furstenberg, Katznelson, and Weiss \cite{FKW90:dist}. Moreover, Bourgain \cite{B86:roth} generalized it to non-degenerate $k$-point patterns, also viewed as vertices of $(k-1)$-dimensional simplices, in $\mathbb{R}^d$ for $d\geq k$.

More recently, Lyall and Magyar \cite{LM16:prod} initiated the consideration of product-type patterns. They proved that, for fixed $a_1,a_2>0$, a positive density subset of $\mathbb{R}^{d_1}\times\mathbb{R}^{d_2}$, $d_1,d_2\geq 2$, contains vertices of a rectangle,
\[ (x_1,x_2), \quad (x_1,x_2+s_2), \quad (x_1+s_1,x_2), \quad (x_1+s_1,x_2+s_2), \]
with $x_1,s_1\in\mathbb{R}^{d_1}$, $x_2,s_2\in\mathbb{R}^{d_2}$, $\|s_1\|_{\ell^2}=\lambda a_1$, and $\|s_2\|_{\ell^2}=\lambda a_2$, for all sufficiently large $\lambda>0$. We write $\|v\|_{\ell^2}$ for the Euclidean norm of a vector $v=(v_1,v_2,\ldots,v_d)\in\mathbb{R}^d$, since later we will also consider more general $\ell^p$-norms, $1<p<\infty$, defined as
\[ \|v\|_{\ell^p} := \big( \sum_{i=1}^{d} |v_i|^p \big)^{1/p}. \]
The particular case $a_1=a_2=1$ corresponds to the search for squares. In the same paper the authors proceed to Cartesian products of two general non-degenerate simplices.

As our first result, we establish a different generalization, replacing (vertices of) rectangles with (vertices of) higher-dimensional rectangular solids. Let $d_1,d_2,\ldots,d_n$ be positive integers. In what follows, a \emph{box} will be a pattern consisting of $2^n$ points in $\mathbb{R}^{d_1}\times\mathbb{R}^{d_2}\times\dots\times\mathbb{R}^{d_n}$ of the form
\begin{equation}\label{eq:box1}
(x_1 + k_1 s_1,\,x_2 + k_2 s_2,\,\dots,\,x_n + k_n s_n), \quad k_1,k_2,\dots,k_n\in\{0,1\}
\end{equation}
for any $x_j,s_j\in\mathbb{R}^{d_j}$, $s_j\neq\mathbf{0}$, $j=1,2,\ldots,n$.

\begin{theorem}\label{thm:main1}
Fix numbers $a_1,a_2,\ldots,a_n>0$. For any positive integers $d_1,d_2,\dots,d_n\geq 5$ and any measurable set $A\subseteq\mathbb{R}^{d_1}\times\mathbb{R}^{d_2}\times\dots\times\mathbb{R}^{d_n}$ with $\overline{\delta}(A)>0$ one can find $\lambda_0>0$ with the property that for any real number $\lambda\geq\lambda_0$ the set $A$ contains a box \eqref{eq:box1} with $x_j,s_j\in\mathbb{R}^{d_j}$ and $\|s_j\|_{\ell^2} = \lambda a_j$, $j=1,2,\ldots,n$.
\end{theorem}

The possibility of generalizing the aforementioned results of Lyall and Magyar to $n$-fold products with $n\geq 3$ was announced by the same authors in \cite{LM16:prod}. The present paper and the more recent preprint by Lyall and Magyar \cite{LM19:hypergraphs} achieve this goal for boxes independently of each other and using quite different approaches.
In fact, \cite{LM19:hypergraphs} proves a sharp variant of Theorem~\ref{thm:main1} above, in which the assumptions $d_j\geq 5$ are relaxed to $d_j\geq 2$. It is clearly necessary to assume $d_j\geq 2$: if we had $d_1=1$, then the set of all points with the first coordinate from
\[ \bigcup_{k\in\mathbb{Z}}\big[(k-1/10)a_1,(k+1/10)a_1\big] \]
would be a counterexample, since it would contain no boxes associated with half-integer values of $\lambda$.

The approach pursued in this paper is in the spirit of the paper by Cook, Magyar, and Pramanik \cite{CMP15:roth}, and the same method will allow us to handle certain enlarged patterns we are about to discuss.

\smallskip
Bourgain \cite{B86:roth} also constructed a measurable set $A\in\mathbb{R}^d$ with $\overline{\delta}(A)>0$ such that lengths $\|s\|_{\ell^2}$ of gaps $s$ for all $3$-term arithmetic progressions
\[ x,\ x+s,\ x+2s \]
inside $A$ omit an unbounded set of positive values. This prevents us from having the most obvious candidate for a density theorem for $3$-term arithmetic progressions. On the other hand, Cook, Magyar, and Pramanik \cite{CMP15:roth} showed that the corresponding density theorem still holds if one is allowed to measure sizes of gaps $s$ in the $\ell^p$-norms for $1<p<\infty$, $p\neq 2$.

Our second result is a common generalization of Theorem~\ref{thm:main1} above and Theorem~2.1 from \cite{CMP15:roth}.
Consider $n$ additional points in $\mathbb{R}^{d_1}\times\mathbb{R}^{d_2}\times\dots\times\mathbb{R}^{d_n}$,
\begin{equation}\label{eq:box2}
(x_1 + 2 s_1, x_2, \dots, x_n),\quad (x_1, x_2 + 2 s_2, \dots, x_n),\quad \ldots,\quad (x_1, x_2, \dots, x_n + 2 s_n)
\end{equation}
for given $x_j,s_j\in\mathbb{R}^{d_j}$, $s_j\neq\mathbf{0}$, $j=1,2,\ldots,n$. The union of \eqref{eq:box1} and \eqref{eq:box2} will be called a \emph{$3$AP-extended box}: it has a $3$-term arithmetic progression attached to each edge coming from a fixed vertex of the box. From the aforementioned observation of Bourgain we know that an analogue of Theorem~\ref{thm:main1} for the $3$AP-extended boxes is not possible, so one has to give up on the Euclidean norm.

\begin{theorem}\label{thm:main2}
Fix numbers $a_1,a_2,\ldots,a_n>0$ and an exponent $1<p<\infty$, $p\neq 2$. There exists a dimensional threshold $d_{\textup{min}}$ such that for any positive integers $d_1,d_2,\dots,$ $d_n\geq d_{\textup{min}}$ and any measurable set $A\subseteq\mathbb{R}^{d_1}\times\mathbb{R}^{d_2}\times\dots\times\mathbb{R}^{d_n}$ with $\overline{\delta}(A)>0$ one can find $\lambda_0>0$ with the property that for any real number $\lambda\geq\lambda_0$ the set $A$ contains a $3$AP-extended box \eqref{eq:box1}$\cup$\eqref{eq:box2} with $x_j,s_j\in\mathbb{R}^{d_j}$ and $\|s_j\|_{\ell^p} = \lambda a_j$, $j=1,2,\ldots,n$.
\end{theorem}

The pattern consisting of points \eqref{eq:box1} and \eqref{eq:box2} can also be viewed as a subset of the grid
\[ (x_1 + k_1 s_1,\,x_2 + k_2 s_2,\,\dots,\,x_n + k_n s_n), \quad k_1,k_2,\dots,k_n\in\{0,1,2\}, \]
consisting of $3^n$ points. At the moment we are not able to prove a result analogous to Theorem~\ref{thm:main2} for this grid. Larger grids bring further complications: one should first handle longer arithmetic progressions and it is known that additional restrictions on the values of $p$ are needed; see the remarks in \cite{DKR17:corner}.

The same approach will enable a further generalization of Theorems~\ref{thm:main1} and \ref{thm:main2}. The present authors and Rimani\'{c} \cite{DKR17:corner} have raised the generality of the result by Cook, Magyar, and Pramanik \cite{CMP15:roth} from $3$-term arithmetic progressions to \emph{corners}, which are triples of points in $\mathbb{R}^d\times\mathbb{R}^d$ of the form
\[ (x,y),\ (x+s,y),\ (x,y+s) \]
for $x,y,s\in\mathbb{R}^d$, $s\neq\mathbf{0}$. A \emph{corner-extended box} will be a pattern in $(\mathbb{R}^{d_1}\times\mathbb{R}^{d_2}\times\dots\times\mathbb{R}^{d_n})^2$ consisting of $2^n$ points forming a box,
\begin{equation}\label{eq:box3}
(x_1 + k_1 s_1,\,x_2 + k_2 s_2,\,\dots,\,x_n + k_n s_n,\,y_1, y_2, \dots, y_n), \quad k_1,k_2,\dots,k_n\in\{0,1\},
\end{equation}
and $n$ additional points completing corners with $n$ of its edges,
\begin{equation}\label{eq:box4}
(x_1, x_2, \dots, x_n,\,y_1 + s_1, y_2, \dots, y_n),\quad \ldots,\quad (x_1, x_2, \dots, x_n,\,y_1, y_2, \dots, y_n + s_n),
\end{equation}
where $x_j,y_j,s_j\in\mathbb{R}^{d_j}$, $s_j\neq\mathbf{0}$, $j=1,2,\ldots,n$. The following result is a common generalization of Theorem~\ref{thm:main1} above and Theorem~1.2 from \cite{DKR17:corner}.

\begin{theorem}\label{thm:main3}
Fix numbers $a_1,a_2,\ldots,a_n>0$ and an exponent $1<p<\infty$, $p\neq 2$. There exists a dimensional threshold $d_{\textup{min}}$ such that for any positive integers $d_1,d_2,\dots,$ $d_n\geq d_{\textup{min}}$ and any measurable set $A\subseteq(\mathbb{R}^{d_1}\times\mathbb{R}^{d_2}\times\dots\times\mathbb{R}^{d_n})^2$ with $\overline{\delta}(A)>0$ one can find $\lambda_0>0$ with the property that for any real number $\lambda\geq\lambda_0$ the set $A$ contains a corner-extended box \eqref{eq:box3}$\cup$\eqref{eq:box4} with $x_j,y_j,s_j\in\mathbb{R}^{d_j}$ and $\|s_j\|_{\ell^p} = \lambda a_j$, $j=1,2,\ldots,n$.
\end{theorem}

Theorem~\ref{thm:main3} implies Theorem~\ref{thm:main2}, as can be seen by considering the skew projections $(x_j,y_j)\mapsto y_j-x_j$; see \cite{DKR17:corner} for details. Consequently, it is still necessary to assume $p\neq 2$, while the endpoint cases $p=1$ and $p=\infty$ clearly do not allow any nontrivial results; see the comments in \cite{CMP15:roth}.

Let us emphasize that the dimensional threshold $d_{\textup{min}}$ in Theorems~\ref{thm:main2} and \ref{thm:main3} depends on the exponent $p$. Careful analysis of the arguments below can give $d_{\textup{min}}=O(p)$ for each fixed $n$, but we do not even have counterexamples to the possibility $d_{\textup{min}}=2$. The optimal value of $d_{\textup{min}}$ is still far from understood, even in the case $n=1$ studied in \cite{CMP15:roth} and \cite{DKR17:corner}.

Lyall and Magyar \cite{LM18:graphs} also worked on the Euclidean embedding of all large dilates of a fixed distance graph. Their results do not include Theorem~\ref{thm:main1}, since the boxes (or even rectangles) are simultaneously ``too rigid'' and ``too degenerate;'' compare with the definition of a \emph{proper $k$-degenerate distance graph} from \cite{LM18:graphs}. They also clearly do not overlap with Theorems~\ref{thm:main2} and \ref{thm:main3}, simply because these theorems fail in the Euclidean metric.

\smallskip
As we have already mentioned, our method of approach is based on the paper by Cook, Magyar, and Pramanik \cite{CMP15:roth}. This method reduces Theorems~\ref{thm:main1}--\ref{thm:main3} to boundedness of certain multilinear singular integral operators. In order to obtain bounds for these operators we invoke the main result from the recent paper by Thiele and one of the present authors \cite{DT18}, which in turn uses techniques gradually developed in a series of papers including \cite{K12:tp}, \cite{D15:L4}, \cite{DKR17:corner}, \cite{DKST17:nvea}, and \cite{DKT16:splx}.

In Section~\ref{sec:scheme} we list the main ingredients of the proofs in the form of several propositions and we explain how they imply the three theorems. Section~\ref{sec:combinatorics} establishes the propositions that belong to the combinatorial part of the proof, by either invoking \cite{CMP15:roth}, or performing necessary modifications. Section~\ref{sec:analysis} establishes the propositions dealing with singular integral operators, which constitute the analytical part of the proof.


\section{Scheme of the proofs}
\label{sec:scheme}

We have already explained how Theorem~\ref{thm:main2} can be derived from Theorem~\ref{thm:main3}, so in this section we give outlines of proofs of Theorems~\ref{thm:main1} and \ref{thm:main3}. In complete analogy with the steps from \cite{CMP15:roth}, they will be reduced to Propositions~\ref{prop:1mainterm}--\ref{prop:3multiscale} below.

If $A,B\colon\mathcal{D}\to[0,\infty)$ are two functions or functionals for which there exists a finite constant $C$ depending on a set of parameters $P$ such that $A(x)\leq CB(x)$ for each $x\in\mathcal{D}$, then we write
\[ A(x)\lesssim_P B(x). \]
If both $A(x)\lesssim_P B(x)$ and $B(x)\lesssim_P A(x)$, then we write
\[ A(x)\sim_P B(x). \]
The parameters in $P$ that are understood throughout the text will be omitted from this notation. In particular, it will always be understood that all constants implicit in the notation $\lesssim$ and $\sim$ depend on the fixed vector of positive numbers
\[ \mathbf{a} := (a_1,a_2,\ldots,a_n), \]
which determines the shape of the patterns (the aspect ratios of the boxes), and the exponent $p$, which is relevant to the proof of Theorem~\ref{thm:main3} only.

Characteristic function (i.e.\@ the indicator function) of a set $E$ will be written as $\mathbbm{1}_E$.
Let us write $g_t$ for an $\textup{L}^1$-normalized dilate of a function $g\colon\mathbb{R}^d\to\mathbb{C}$ by a factor $t>0$, i.e.,
\begin{equation}\label{eq:funcdilate}
g_t(s) := t^{-d}g(t^{-1}s)
\end{equation}
for each $s\in\mathbb{R}^d$.
The Fourier transform of $\textup{L}^1$ functions is normalized as
\[ \widehat{g}(\xi) := \int_{\mathbb{R}^d} g(s) e^{-2\pi i s\cdot\xi} \,\textup{d}s, \]
where $s\cdot\xi$ stands the standard scalar product of vectors $s$ and $\xi$ in $\mathbb{R}^d$.
If $\sigma$ is a measure on Borel subsets of $\mathbb{R}^d$, then we define its dilate by $t>0$ as another measure $\sigma_t$ given as
\begin{equation}\label{eq:measdilate}
\sigma_{t}(E):=\sigma(t^{-1}E)
\end{equation}
for each Borel set $E\subseteq\mathbb{R}^d$. A consequence of a linear change of variables is
\[ \int_{\mathbb{R}^d} f(s) \,\textup{d}\sigma_{t}(s) = \int_{\mathbb{R}^d} f(ts) \,\textup{d}\sigma(s) \]
for any measurable function $f\colon\mathbb{R}^d\to\mathbb{C}$ such that the above integrals exist.
Notation \eqref{eq:funcdilate} and \eqref{eq:measdilate} is mutually consistent when $\sigma$ is absolutely continuous with respect to the Lebesgue measure with density $g$. Occasionally we will need an $\textup{L}^p$-normalized dilate of $g\colon\mathbb{R}^d\to\mathbb{C}$ by $t>0$, for a more general exponent $1\leq p<\infty$, which will be denoted $\textup{D}_t^{p}g$ and defined as
\begin{equation}\label{eq:funcdilatep}
\textup{D}_t^{p}g(s) := t^{-d/p}g(t^{-1}s).
\end{equation}

Let us fix an exponent $1<p<\infty$; it will simply be $p=2$ in relation with Theorem~\ref{thm:main1}, while the proof of Theorem~\ref{thm:main3} will assume $p\neq 2$. We introduce a measure $\sigma^{d,p}$ on Borel subsets of $\mathbb{R}^{d}$ in the Dirac $\delta$ notation as
\[ \sigma^{d,p}(s) := \delta\big(1-\|s\|_{\ell^p}^p\big) \]
or, less formally and abusing the integral representation for the Fourier transform, as
\[ \sigma^{d,p}(s) = \int_{\mathbb{R}} e^{-2\pi i u (1-\|s\|_{\ell^p}^p)} \,\textup{d}u. \]
Its dilate $\sigma^{d,p}_{\lambda}$ by $\lambda>0$ is clearly supported on the $\textup{C}^1$ surface $\{s\in\mathbb{R}^d : \|s\|_{\ell^p} = \lambda\}$.
Let us also fix a Schwartz function $\psi\colon\mathbb{R}\to[0,1]$ such that $\widehat{\psi}\geq 0$, $\widehat{\psi}$ is supported in $[-4,4]$, $\widehat{\psi}(1)>0$, and $\psi(0)=1$. For instance, we can take a $\textup{C}^\infty$ function $\rho\colon\mathbb{R}\to[0,\infty)$ such that $\rho>0$ on $[-1,1]$, $\rho=0$ outside $[-2,2]$, and $\rho$ has integral $1$; then we can simply set $\psi=|\widehat{\rho}|^2$. Any constants implicit in the notation $\lesssim$ and $\sim$ will also be understood to depend on $\psi$. Furthermore, for any $\varepsilon>0$ we introduce a function $\omega^{d,p,\varepsilon}\colon\mathbb{R}^d\to\mathbb{C}$ by the formula
\[ \omega^{d,p,\varepsilon}(s) := \int_{\mathbb{R}} e^{-2\pi i u (1-\|s\|_{\ell^p}^p)} \psi(\varepsilon u) \,\textup{d}u
= \varepsilon^{-1}\widehat{\psi}\big(\varepsilon^{-1}(1-\|s\|_{\ell^p}^p)\big). \]
It was shown in Lemma~4.1 of \cite{CMP15:roth} that
\[ \int_{\mathbb{R}^d} \omega^{d,p,\varepsilon}(s) \,\textup{d}s \sim_{d,p} 1 \]
for all $0<\varepsilon<1/10d$. Thus, for such $\varepsilon$ we set
\begin{equation}\label{eq:omegaquotient}
c(d,p,\varepsilon) := \frac{\int_{\mathbb{R}^d} \omega^{d,p,\varepsilon}(s) \,\textup{d}s}{\int_{\mathbb{R}^d} \omega^{d,p,1}(s) \,\textup{d}s} \sim_{d,p} 1
\end{equation}
and then $k^{d,p,\varepsilon}\colon\mathbb{R}^d\to\mathbb{R}$ defined by
\begin{equation}\label{eq:defofk}
k^{d,p,\varepsilon}(s) := \omega^{d,p,\varepsilon}(s) - c(d,p,\varepsilon) \omega^{d,p,1}(s)
\end{equation}
has integral equal to $0$.

We introduce the number $D=d_1+\cdots+d_n$, so that
\[ \mathbb{R}^{D}\cong\mathbb{R}^{d_1}\times\dots\times\mathbb{R}^{d_n}. \]
Throughout the proofs we will use the shorthand notation
\[ \mathbf{x} := (x_1, \dots, x_n),\quad \mathbf{y} := (y_1, \dots, y_n),\quad \mathbf{s} := (s_1, \dots, s_n) \]
and we view $\mathbf{x}$, $\mathbf{y}$, and $\mathbf{s}$ as vectors from $\mathbb{R}^{D}$.
It will also be convenient to adopt some derived notation, such as
\[ \textup{d}\mathbf{x} := \textup{d}x_1 \cdots \textup{d}x_n, \quad
\textup{d}\sigma^{p}_{\lambda\mathbf{a}}(\mathbf{s}) := \textup{d}\sigma^{d_1,p}_{\lambda a_1}(s_1) \cdots \textup{d}\sigma^{d_n,p}_{\lambda a_n}(s_n). \]
In the same spirit we define
\[ \omega^{p,\varepsilon}_{\lambda\mathbf{a}}(\mathbf{s}) := \omega^{d_1,p,\varepsilon}_{\lambda a_1}(s_1) \cdots \omega^{d_n,p,\varepsilon}_{\lambda a_n}(s_n). \]
Let us also write $\mathbf{k}=(k_1,\ldots,k_n)\in\{0,1\}^n$ and denote
\[ (\mathcal{F}f)(\mathbf{x},\mathbf{s}) := \prod_{\mathbf{k}\in\{0,1\}^n} f(x_1 + k_1 s_1,\,\dots,\,x_n + k_n s_n) \]
for a function $f\colon\mathbb{R}^D\to[0,1]$ and
\begin{align*}
(\widetilde{\mathcal{F}}f)(\mathbf{x},\mathbf{y},\mathbf{s}) :=
& \Big( \prod_{\mathbf{k}\in\{0,1\}^n} f(x_1 + k_1 s_1,\,\dots,\,x_n + k_n s_n,\,y_1, \dots, y_n) \Big) \\
& f(x_1, \dots, x_n,\,y_1 + s_1, \dots, y_n) \cdots f(x_1, \dots, x_n,\,y_1, \dots, y_n + s_n)
\end{align*}
for a function $f\colon(\mathbb{R}^{D})^2\to[0,1]$.

The most important objects are the \emph{pattern-counting forms}, defined as follows. For a ``scale'' $\lambda>0$ we set
\[ \mathcal{N}^{p}_{\lambda}(f) := \int_{(\mathbb{R}^{D})^2} (\mathcal{F}f)(\mathbf{x},\mathbf{s}) \,\textup{d}\sigma^{p}_{\lambda\mathbf{a}}(\mathbf{s}) \,\textup{d}\mathbf{x} \]
and
\[ \widetilde{\mathcal{N}}^{p}_{\lambda}(f) := \int_{(\mathbb{R}^{D})^3} (\widetilde{\mathcal{F}}f)(\mathbf{x},\mathbf{y},\mathbf{s}) \,\textup{d}\sigma^{p}_{\lambda\mathbf{a}}(\mathbf{s}) \,\textup{d}\mathbf{x} \,\textup{d}\mathbf{y}. \]
The name comes from the fact that if $\mathcal{N}^{p}_{\lambda}(\mathbbm{1}_A)>0$ (resp.\@ $\widetilde{\mathcal{N}}^{p}_{\lambda}(\mathbbm{1}_A)>0$), then $A$ contains a box \eqref{eq:box1} (resp.\@ a corner-extended box \eqref{eq:box3}$\cup$\eqref{eq:box4}) with $\|s_j\|_{\ell^p} = \lambda a_j$, $j=1,2,\ldots,n$.
We will also need their smoothened versions, defined for $\varepsilon>0$ as
\[ \mathcal{M}^{p,\varepsilon}_{\lambda}(f) := \int_{(\mathbb{R}^{D})^2} (\mathcal{F}f)(\mathbf{x},\mathbf{s}) \omega^{p,\varepsilon}_{\lambda\mathbf{a}}(\mathbf{s}) \,\textup{d}\mathbf{s} \,\textup{d}\mathbf{x} \]
and
\[ \widetilde{\mathcal{M}}^{p,\varepsilon}_{\lambda}(f) := \int_{(\mathbb{R}^{D})^3} (\widetilde{\mathcal{F}}f)(\mathbf{x},\mathbf{y},\mathbf{s}) \omega^{p,\varepsilon}_{\lambda\mathbf{a}}(\mathbf{s}) \,\textup{d}\mathbf{s} \,\textup{d}\mathbf{x} \,\textup{d}\mathbf{y}. \]
By the standard approximation of identity arguments,
\begin{equation}\label{eq:limits}
\lim_{\varepsilon\to0^+} \mathcal{M}^{p,\varepsilon}_{\lambda}(f) = \mathcal{N}^{p}_{\lambda}(f), \quad
\lim_{\varepsilon\to0^+} \widetilde{\mathcal{M}}^{p,\varepsilon}_{\lambda}(f) = \widetilde{\mathcal{N}}^{p}_{\lambda}(f)
\end{equation}
for functions $f$ as above.
Finally, we denote
\[ \mathcal{E}^{p,\varepsilon}_{\lambda}(f) := \mathcal{M}^{p,\varepsilon}_{\lambda}(f) - b(p,\varepsilon) \mathcal{M}^{p,1}_{\lambda}(f) \]
and
\[ \widetilde{\mathcal{E}}^{p,\varepsilon}_{\lambda}(f) := \widetilde{\mathcal{M}}^{p,\varepsilon}_{\lambda}(f) - b(p,\varepsilon) \widetilde{\mathcal{M}}^{p,1}_{\lambda}(f), \]
where we recall that the numbers $c(d_j,p,\varepsilon)$ come from \eqref{eq:omegaquotient} and use the shorthand notation
\[ b(p,\varepsilon) := c(d_1,p,\varepsilon) \cdots c(d_n,p,\varepsilon). \]

Here are the three main propositions needed in the proofs.

\begin{proposition}\label{prop:1mainterm}
Suppose that $1<p<\infty$ and that $\delta,\lambda,N$ are real numbers such that $0<\delta\leq 1$ and $0<\lambda\leq N$.
\begin{itemize}
\item[(a)] If $f\colon\mathbb{R}^{D}\to[0,1]$ is a measurable function supported in $[0,N]^D$ and satisfying\linebreak $\int_{[0,N]^D} f \geq \delta N^D$, then
\[ \mathcal{M}^{p,1}_{\lambda}(f) \gtrsim_{D,\delta} N^D. \]
\item[(b)] If $f\colon\mathbb{R}^{2D}\to[0,1]$ is a measurable function supported in $[0,N]^{2D}$ and satisfying $\int_{[0,N]^{2D}} f \geq \delta N^{2D}$, then
\[ \widetilde{\mathcal{M}}^{p,1}_{\lambda}(f) \gtrsim_{D,\delta} N^{2D}. \]
\end{itemize}
\end{proposition}

\begin{proposition}\label{prop:2difference}
Suppose that $\varepsilon,\lambda,N$ are real numbers such that $0<\varepsilon<1$ and $0<\lambda\leq N$.
\begin{itemize}
\item[(a)] If $d_j\geq 5$ for $j=1,2,\ldots,n$ and if $f\colon\mathbb{R}^{D}\to[0,1]$ is a measurable function supported in $[0,N]^D$, then
\[ \big| \mathcal{N}^{2}_{\lambda}(f) - \mathcal{M}^{2,\varepsilon}_{\lambda}(f) \big| \lesssim_D \varepsilon^{1/4} N^D. \]
\item[(b)] Additionally, take $1<p<\infty$, $p\neq 2$. If each $d_j$ is sufficiently large for $j=1,2,\ldots,n$ and if $f\colon\mathbb{R}^{2D}\to[0,1]$ is a measurable function supported in $[0,N]^{2D}$, then
\[ \big| \widetilde{\mathcal{N}}^{p}_{\lambda}(f) - \widetilde{\mathcal{M}}^{p,\varepsilon}_{\lambda}(f) \big| \lesssim_D \varepsilon N^{2D}. \]
\end{itemize}
\end{proposition}

\begin{proposition}\label{prop:3multiscale}
Suppose that $1<p<\infty$, $0<\varepsilon<1/10D$, and that $\lambda_1,\lambda_2,\ldots,\lambda_M$ are positive numbers such that $\lambda_{m+1}/\lambda_{m}\geq 2$ for $m=1,2,\ldots,M-1$.
\begin{itemize}
\item[(a)] If $f\colon\mathbb{R}^{D}\to[0,1]$ is a measurable function supported in $[0,N]^D$, then
\[ \sum_{m=1}^{M} |\mathcal{E}^{p,\varepsilon}_{\lambda_m}(f)| \lesssim_{D,\varepsilon} N^D. \]
\item[(b)] If $f\colon\mathbb{R}^{2D}\to[0,1]$ is a measurable function supported in $[0,N]^{2D}$, then
\[ \Big( \sum_{m=1}^{M} |\widetilde{\mathcal{E}}^{p,\varepsilon}_{\lambda_m}(f)|^2 \Big)^{1/2} \lesssim_{D,\varepsilon} N^{2D}. \]
\end{itemize}
\end{proposition}

Proofs of Propositions~\ref{prop:1mainterm}--\ref{prop:3multiscale} are postponed to the later sections. Now we show how they imply Theorems~\ref{thm:main1} and \ref{thm:main3}.

\begin{proof}[Proof of Theorem~\ref{thm:main1}]
We argue by contradiction and suppose that there is a set $A$ with strictly positive upper Banach density $\overline{\delta}(A)$ for which the claim does not hold: there exists a sequence $(\lambda_m)_{m=1}^{\infty}$ such that $\lim_{m\to\infty}\lambda_m=\infty$ and that, for each $m$, the set $A$ contains no boxes \eqref{eq:box1} with $\|s_j\|_{\ell^2} = \lambda_m a_j$, $j=1,2,\ldots,n$. By omitting some terms we can achieve $\lambda_{m+1}/\lambda_{m}\geq 2$ for each $m$. Fix an arbitrary positive integer $M$.
For $\delta:=\overline{\delta}(A)/2>0$, by the definition of the upper Banach density \eqref{eq:defofdensity}, we can find $N\geq\lambda_M$ and $\mathbf{x}\in\mathbb{R}^D$ such that
\[ \big| A\cap(\mathbf{x}+[0,N]^D) \big| \geq \delta N^D. \]
Denote $A':=(-\mathbf{x}+A)\cap[0,N]^D$ and $f=\mathbbm{1}_{A'}$, so that $A'$ is now a measurable subset of $[0,N]^D$ satisfying
\begin{equation}\label{eq:thmgrtdelta}
\int_{[0,N]^D}f(\mathbf{x})\,\textup{d}\mathbf{x} = |A'| \geq \delta N^D
\end{equation}
and it still has no boxes \eqref{eq:box1} of the previously described sizes determined by $\lambda$.
Consequently,
\begin{equation}\label{eq:thmaux1}
\mathcal{N}^{2}_{\lambda_m}(f)=0
\end{equation}
for $m=1,2,\ldots,M$. Because of condition \eqref{eq:thmgrtdelta}, we can apply part (a) of Proposition~\ref{prop:1mainterm} and get
\begin{equation}\label{eq:thmaux2}
\mathcal{M}^{2,1}_{\lambda_m}(f) \gtrsim_{D,\delta} N^D,
\end{equation}
again for each $m=1,2,\ldots,M$. Moreover, \eqref{eq:thmaux1}, part (a) of Proposition~\ref{prop:2difference} and \eqref{eq:thmaux2} together give
\begin{equation}\label{eq:thmaux3}
\mathcal{M}^{2,\varepsilon}_{\lambda_m}(f) = \big| \mathcal{N}^{2}_{\lambda_m}(f) - \mathcal{M}^{2,\varepsilon}_{\lambda_m}(f) \big|
\lesssim_D \varepsilon^{1/4} N^D \lesssim_{D,\delta} \varepsilon^{1/4} \mathcal{M}^{2,1}_{\lambda_m}(f).
\end{equation}
By \eqref{eq:thmaux3} and \eqref{eq:omegaquotient}, for a sufficiently small $\varepsilon$ depending on the dimensions and $\delta$, we have
\[ \mathcal{M}^{2,\varepsilon}_{\lambda_m}(f) \leq \frac{1}{2} b(2,\varepsilon) \mathcal{M}^{2,1}_{\lambda_m}(f), \]
so,
\[ |\mathcal{E}^{2,\varepsilon}_{\lambda_m}(f)| = b(2,\varepsilon) \mathcal{M}^{2,1}_{\lambda_m}(f) - \mathcal{M}^{2,\varepsilon}_{\lambda_m}(f)
\geq \frac{1}{2} b(2,\varepsilon) \mathcal{M}^{2,1}_{\lambda_m}(f). \]
By \eqref{eq:omegaquotient} and \eqref{eq:thmaux2} again, we conclude
\begin{equation}\label{eq:thmaux4}
|\mathcal{E}^{2,\varepsilon}_{\lambda_m}(f)| \gtrsim_{D,\delta} N^D
\end{equation}
for each $m=1,2,\ldots,M$. Summing the lower bound \eqref{eq:thmaux4} in $m$ gives
\begin{equation}\label{eq:thmaux5}
\sum_{m=1}^{M} |\mathcal{E}^{2,\varepsilon}_{\lambda_m}(f)| \gtrsim_{D,\delta} M N^D.
\end{equation}
Finally, combining \eqref{eq:thmaux5} with Proposition~\ref{prop:3multiscale} yields $M\lesssim_{D,\delta}1$, which contradicts the fact that $M$ could have been chosen arbitrarily large.
\end{proof}

\begin{proof}[Proof of Theorem~\ref{thm:main3}]
The same outline also applies here. The only difference is that part (b) of Proposition~\ref{prop:2difference} only holds for sufficiently large dimensions $d_j$ depending on $p$.
The reader can also consult the corresponding proofs of Theorem~2.2 in \cite{CMP15:roth} and Theorem~1.2 in \cite{DKR17:corner}.
\end{proof}


\section{Combinatorial results}
\label{sec:combinatorics}

The following lemma is needed in the proof of Proposition~\ref{prop:1mainterm} in the same way in which Bourgain's version of Roth's theorem for compact abelian groups \cite{B86:roth} is needed in the analogous proposition in \cite{CMP15:roth}.

\begin{lemma}\label{lm:1combroth}
Suppose $0<\delta\leq 1$.
\begin{itemize}
\item[(a)] If $f\colon\mathbb{R}^{D}\to[0,1]$ is a measurable function supported in $[0,N]^D$ and satisfying\linebreak $\int_{[0,N]^D} f \geq \delta N^D$, then
\[ \int_{([0,1]^{D})^2} (\mathcal{F}f)(\mathbf{x},\mathbf{s}) \,\textup{d}\mathbf{s} \,\textup{d}\mathbf{x} \gtrsim_{D,\delta} 1. \]
\item[(b)] If $f\colon\mathbb{R}^{2D}\to[0,1]$ is a measurable function supported in $[0,N]^{2D}$ and satisfying $\int_{[0,N]^{2D}} f \geq \delta N^{2D}$, then
\[ \int_{([0,1]^{D})^3} (\widetilde{\mathcal{F}}f)(\mathbf{x},\mathbf{y},\mathbf{s}) \,\textup{d}\mathbf{s} \,\textup{d}\mathbf{x} \,\textup{d}\mathbf{y} \gtrsim_{D,\delta} 1. \]
\end{itemize}
\end{lemma}

\begin{proof}[Proof of Lemma~\ref{lm:1combroth}]
Both parts of the lemma are shown using multidimensional Szemer\'{e}di's theorem of Furstenberg and Katznelson \cite{FK78:msz}.
By this result, for any dimension $n$ and any number $0<\beta\leq 1$ there exists a positive integer $m_{n,\beta}$ such that for each positive integer $m\geq m_{n,\beta}$ one has the following.
\begin{itemize}
\item Each subset $S\subseteq\{0,1,\ldots,m-1\}^n$ of cardinality at least $\beta m^n$ contains (vertices of) an $n$-dimensional cube,
\[ (i_1 + k_1 l,\,i_2 + k_2 l,\,\dots,\,i_n + k_n l), \quad k_1,k_2,\dots,k_n\in\{0,1\} \]
for some $i_1,i_2,\ldots,i_n,l\in\mathbb{Z}$ with $l\neq 0$.
\item Each subset $S\subseteq\{0,1,\ldots,m-1\}^{2n}$ of cardinality at least $\beta m^{2n}$ contains a $2n$-dimensional corner-extended cube,
\begin{align*}
(i_1 + k_1 l,\,i_2 + k_2 l,\,\dots,\,i_n + k_n l,\,j_1, j_2, \dots, j_n), \quad k_1,k_2,\dots,k_n\in\{0,1\}, \\
(i_1, i_2, \dots, i_n,\,j_1 + l, j_2, \dots, j_n),\quad \ldots,\quad (i_1, i_2, \dots, i_n,\,j_1, j_2, \dots, j_n + l)
\end{align*}
for some $i_1,i_2,\ldots,i_n,j_1,j_2,\ldots,j_n,l\in\mathbb{Z}$ with $l\neq 0$.
\end{itemize}
Then one applies the averaging trick of Varnavides \cite{Var59:dens}, in the same way it was done in the proof of Lemma~3.2 in \cite{DKR17:corner} for the particular case of the three-point corners.
\end{proof}

Indeed, multidimensional Szemer\'{e}di's theorem applies to any finite pattern on the integer lattice, not only to boxes and corner-extended boxes, so Lemma~\ref{lm:1combroth} can be generalized easily. The reasons why we restrict our attention to very special patterns lie in the rigidity of other auxiliary results, most notably Proposition~\ref{prop:3multiscale} above and Theorem~\ref{thm:singular} from Section~\ref{sec:analysis}.

\begin{proof}[Proof of Proposition~\ref{prop:1mainterm}]
The proposition is shown by cutting the Euclidean space into cubes, the scaled copies of $([0,1]^{D})^2$ or $([0,1]^{D})^3$, and applying Lemma~\ref{lm:1combroth} on each of them. For details the reader can consult the proof of Proposition~2.1 in \cite{CMP15:roth}.
\end{proof}

Now we turn to the proof of the second proposition. We will need the Euclidean version of the notion of the Gowers norms, so let us begin by setting
\[ (\Delta_{h}g)(s) := g(s) \overline{g(s+h)} \]
for $s,h\in\mathbb{R}^d$ and a function $g\colon\mathbb{R}^d\to\mathbb{C}$. If such $g$ is also measurable, then its \emph{Gowers uniformity norm of degree $k$} is defined as
\begin{align}
\|g\|_{\textup{U}^k(\mathbb{R}^d)} & := \Big( \int_{(\mathbb{R}^d)^{k+1}} (\Delta_{h_k} \cdots \Delta_{h_1} g)(s) \,\textup{d}s \,\textup{d}h_1 \cdots \,\textup{d}h_k \Big)^{2^{-k}} \nonumber \\
& = \Big( \int_{(\mathbb{R}^d)^{k-1}} \Big| \int_{\mathbb{R}^d} (\Delta_{h_{k-1}} \cdots \Delta_{h_1} g)(s) \,\textup{d}s \Big|^2 \,\textup{d}h_1 \cdots \,\textup{d}h_{k-1} \Big)^{2^{-k}} \label{eq:gowersnorm}
\end{align}
We will only need the norms $\|\cdot\|_{\textup{U}^2(\mathbb{R}^d)}$ and $\|\cdot\|_{\textup{U}^3(\mathbb{R}^d)}$. The Gowers norms scale properly with respect to the $\textup{L}^1$-normalized dilations of the function. In particular,
\begin{equation}\label{eq:gowersscaling}
\|g_t\|_{\textup{U}^2(\mathbb{R}^d)} = t^{-d/4} \|g\|_{\textup{U}^2(\mathbb{R}^d)}, \quad
\|g_t\|_{\textup{U}^3(\mathbb{R}^d)} = t^{-d/2} \|g\|_{\textup{U}^3(\mathbb{R}^d)},
\end{equation}
as is shown by an easy change of variables of integration.

\begin{lemma}\label{lm:2combcs}
Suppose that $\lambda$ and $N$ are real numbers such that $0<\lambda\leq N$.
\begin{itemize}
\item[(a)] If $f_1,f_2\colon\mathbb{R}^{d}\to[0,1]$ are measurable functions supported in $[0,N]^d$, and $g\colon\mathbb{R}^{d}\to\mathbb{R}$ is a measurable function supported in $[-3\lambda,3\lambda]^d$, then
\begin{equation}\label{eq:lemmanewa}
\Big| \int_{(\mathbb{R}^{d})^2} f_1(x) f_2(x+s) g(s) \,\textup{d}s \,\textup{d}x \Big| \lesssim_d N^{d} \lambda^{d/4} \|g\|_{\textup{U}^2(\mathbb{R}^d)}.
\end{equation}
\item[(b)] If $f_1,f_2,f_3\colon\mathbb{R}^{2d}\to[0,1]$ are measurable functions supported in $[0,N]^{2d}$ and $g\colon\mathbb{R}^{d}\to\mathbb{R}$ is a measurable function supported in $[-3\lambda,3\lambda]^d$, then
\begin{equation}\label{eq:lemmanewb}
\Big| \int_{(\mathbb{R}^{d})^3} f_1(x,y) f_2(x+s,y) f_3(x,y+s) g(s) \,\textup{d}s \,\textup{d}x \,\textup{d}y \Big| \lesssim_d N^{2d} \lambda^{d/2} \|g\|_{\textup{U}^3(\mathbb{R}^d)}.
\end{equation}
\end{itemize}
\end{lemma}

\begin{proof}[Proof of Lemma~\ref{lm:2combcs}]
This lemma will be shown in a way similar to the proof of Lemma 4.2 in \cite{CMP15:roth}.

Proof of (a).
The left hand side of \eqref{eq:lemmanewa} can be rewritten as
\[ \Big| \int_{[0,N]^d} f_1(x) \Big( \int_{\mathbb{R}^d} f_2(x+s) g(s) \,\textup{d}s \Big) \,\textup{d}x \Big|, \]
so the Cauchy--Schwarz inequality in $x$ bounds it by
\[ N^{d/2} \bigg( \int_{(\mathbb{R}^d)^3} f_2(x+s) f_2(x+s') g(s) g(s') \,\textup{d}s \,\textup{d}s' \,\textup{d}x \bigg)^{1/2}. \]
Substituting $y=x+s$, $h=s'-s$, applying the Cauchy--Schwarz inequality again, and recognizing the $\textup{U}^2$-norm from \eqref{eq:gowersnorm}, we get
\[ N^{d/2} \|g\|_{\textup{U}^2(\mathbb{R}^d)} \bigg( \int_{[-6\lambda,6\lambda]^d} \Big( \int_{[0,N]^d} f_2(y) f_2(y+h) \,\textup{d}y \Big)^2 \,\textup{d}h \bigg)^{1/4}. \]
The expression within the outer parentheses is clearly at most a constant times $N^{2d}\lambda^{d}$.

Proof of (b).
Using the Cauchy--Schwarz inequality we first bound the left hand side of \eqref{eq:lemmanewb} by
\[ N^d \bigg( \int_{(\mathbb{R}^{d})^4} f_2(x+s,y) f_2(x+s',y) f_3(x,y+s) f_3(x,y+s') g(s) g(s') \,\textup{d}s \,\textup{d}s' \,\textup{d}x \,\textup{d}y \bigg)^{1/2}. \]
Then we substitute $h=s'-s$ and perform the shift $x\mapsto x-s$ to transform the integral inside parentheses into
\[ \int_{(\mathbb{R}^{d})^4} f_2(x,y) f_2(x+h,y) f_3(x-s,y+s) f_3(x-s,y+s+h) g(s) g(s+h) \,\textup{d}x \,\textup{d}y \,\textup{d}s \,\textup{d}h. \]
Another application of the Cauchy--Schwarz inequality controls the left hand side of \eqref{eq:lemmanewb} with
\begin{align*}
N^{3d/2} \lambda^{d/4} \bigg( \int_{(\mathbb{R}^{d})^5} & f_3(x-s,y+s) f_3(x-s',y+s') f_3(x-s,y+s+h) f_3(x-s',y+s'+h) \\
& g(s) g(s') g(s+h) g(s'+h) \,\textup{d}s \,\textup{d}s' \,\textup{d}x \,\textup{d}y \,\textup{d}h \bigg)^{1/4}.
\end{align*}
It remains to substitute $h'=s'-s$, shift $x\mapsto x+s$, $y\mapsto y-s$, use the Cauchy--Schwarz inequality one more time, and finally recognize $\|g\|_{\textup{U}^3(\mathbb{R}^d)}$ from \eqref{eq:gowersnorm}.
\end{proof}

\begin{lemma}\label{lm:3combgowers}
Suppose that $\eta$ and $\varepsilon$ are real numbers such that $0<\eta<\varepsilon<1$.
\begin{itemize}
\item[(a)] For $d\geq 5$ we have
\[ \big\| \omega^{d,2,\eta} - \omega^{d,2,\varepsilon} \big\|_{\textup{U}^2(\mathbb{R}^d)} \lesssim_{d} \varepsilon^{1/4}. \]
\item[(b)] For $1<p<\infty$, $p\neq 2$ and sufficiently large $d$ depending on $p$ we have
\[ \big\| \omega^{d,p,\eta} - \omega^{d,p,\varepsilon} \big\|_{\textup{U}^3(\mathbb{R}^d)} \lesssim_{p,d} \varepsilon. \]
\end{itemize}
\end{lemma}

\begin{proof}[Proof of Lemma~\ref{lm:3combgowers}]
Part (b) was already established in \cite{CMP15:roth}; it is Lemma~2.4 of that paper.

We will show part (a) using the same lines of proof, but we need to use to our advantage the fact that we only need the $\textup{U}^2$-norm and get more concrete decay in $\varepsilon$ as $\varepsilon\to0^+$.
Let $\varphi\colon\mathbb{R}\to[0,\infty)$ be a compactly supported $\textup{C}^\infty$ function that is constantly equal to $1$ on $[-3,3]$ and set
\[ \Phi = \underbrace{\varphi\otimes\cdots\otimes\varphi}_{d}. \]
All constants are assumed to depend on $\varphi$ without further mention.
Observe that $\omega^{d,2,\eta}$ and $\omega^{d,2,\varepsilon}$ are supported on $[-3,3]^d$, so
\[ \omega^{d,2,\eta}(s) - \omega^{d,2,\varepsilon}(s)
= \int_{\mathbb{R}} ( \psi(\eta u) - \psi(\varepsilon u) ) \Phi(s) e^{2\pi i u (\|s\|_{\ell^2}^2-1)} \,\textup{d}u. \]
Integral version of the triangle inequality for the Gowers norm and the tensor product splitting of the exponential give
\begin{equation}\label{eq:u2aux1}
\big\| \omega^{d,2,\eta} - \omega^{d,2,\varepsilon} \big\|_{\textup{U}^2(\mathbb{R}^d)}
\leq \int_{\mathbb{R}} | \psi(\eta u) - \psi(\varepsilon u) | \big\| \varphi(s) e^{2\pi i u s^2} \big\|_{\textup{U}^2_{s}(\mathbb{R})}^d \,\textup{d}u.
\end{equation}
By definition of the Gowers norm \eqref{eq:gowersnorm},
\[ \big\| \varphi(s) e^{2\pi i u s^2} \big\|_{\textup{U}^2_{s}(\mathbb{R})}^4
= \int_{[-3,3]} \Big| \int_{\mathbb{R}} \varphi(s) \varphi(s+h) e^{-4 \pi i u h s} \,\textup{d}s \Big|^2 \,\textup{d}h, \]
which is certainly bounded by a constant, for each $u\in\mathbb{R}$. However, for $|u|\geq 1$ we get a better estimate by splitting the outer domain of integration into $|h|\leq|u|^{-1}$ and $|u|^{-1}<|h|\leq 3$. The first part of the integral is clearly at most a constant times $|u|^{-1}$. Integration by parts in the second part gives
\[ \Big| \int_{\mathbb{R}} \varphi(s) \varphi(s+h) e^{-4 \pi i u h s} u h \,\textup{d}s \Big| \lesssim 1, \]
so that
\[ \int_{\{|u|^{-1}<|h|\leq3\}} \Big| \int_{\mathbb{R}} \varphi(s) \varphi(s+h) e^{-4 \pi i u h s} \,\textup{d}s \Big|^2 \,\textup{d}h
\lesssim |u|^{-2} \int_{\{|u|^{-1}<|h|\leq3\}} h^{-2} \,\textup{d}h \lesssim |u|^{-1}. \]
From these we conclude
\begin{equation}\label{eq:u2aux2}
\big\| \varphi(s) e^{2\pi i u s^2} \big\|_{\textup{U}^2_{s}(\mathbb{R})} \lesssim \min\{1,|u|^{-1/4}\}.
\end{equation}

Now we combine \eqref{eq:u2aux1} and \eqref{eq:u2aux2} into a single estimate
\[ \big\| \omega^{d,2,\eta} - \omega^{d,2,\varepsilon} \big\|_{\textup{U}^2(\mathbb{R}^d)}
\lesssim \int_{\mathbb{R}} | \psi(\eta u) - \psi(\varepsilon u) | \min\{1,|u|^{-d/4}\} \,\textup{d}u. \]
This time we split the domain of integration into three parts: $|u|\leq 1$, $1<|u|\leq\varepsilon^{-1}$, and $|u|>\varepsilon^{-1}$.
We bound the corresponding integrals respectively as
\[ \int_{\{|u|\leq 1\}} \|\psi'\|_{\textup{L}^\infty(\mathbb{R})} \varepsilon |u| \,\textup{d}u \lesssim \varepsilon \lesssim \varepsilon^{1/4}, \]
\[ \int_{\{1<|u|\leq\varepsilon^{-1}\}} \|\psi'\|_{\textup{L}^\infty(\mathbb{R})} \varepsilon |u|^{1-d/4} \,\textup{d}u \lesssim_d \varepsilon^{d/4-1} \lesssim \varepsilon^{1/4}, \]
and
\[ \int_{\{|u|>\varepsilon^{-1}\}} 2\|\psi\|_{\textup{L}^\infty(\mathbb{R})} |u|^{-d/4} \,\textup{d}u \lesssim_d \varepsilon^{d/4-1} \lesssim \varepsilon^{1/4}. \]
In the last display we needed $d>4$ for the convergence of the improper integral and also to have $d/4-1\geq 1/4>0$.
This allows us to conclude the desired inequality.
\end{proof}

\begin{proof}[Proof of Proposition~\ref{prop:2difference}]
Both parts of the proposition are shown in exactly the same way, using Lemmata~\ref{lm:2combcs} and \ref{lm:3combgowers}, so we only elaborate on the proof of part (a).

Because of \eqref{eq:limits} it is enough to bound the difference
\begin{equation}\label{eq:diffaux}
\mathcal{M}^{2,\eta}_{\lambda}(f)-\mathcal{M}^{2,\varepsilon}_{\lambda}(f)
\end{equation}
for all $0<\eta<\varepsilon<1$, with a constant independent of $\eta$. The difference of the corresponding cutoff functions can be expanded as
\[ \omega^{p,\eta}_{\lambda\mathbf{a}}(\mathbf{s}) - \omega^{p,\varepsilon}_{\lambda\mathbf{a}}(\mathbf{s})
= \sum_{j=1}^{n} \tau^{(j)}_{\lambda}(\mathbf{s}), \]
where
\[ \tau^{(j)}(\mathbf{s}) := \Big( \prod_{i=1}^{j-1} \omega^{d_i,p,\eta}_{a_i}(s_i) \Big)
\big( \omega^{d_j,p,\eta}_{a_j}(s_j) - \omega^{d_j,p,\varepsilon}_{a_j}(s_j) \big)
\Big( \prod_{i=j+1}^{n} \omega^{d_i,p,\varepsilon}_{a_i}(s_i) \Big). \]
This decomposes \eqref{eq:diffaux} into $n$ pieces, $\sum_{j=1}^{n}\mathcal{P}^{(j)}_{\lambda}(f)$, where
\[ \mathcal{P}^{(j)}_{\lambda}(f) := \int_{(\mathbb{R}^{D})^2} (\mathcal{F}f)(\mathbf{x},\mathbf{s}) \tau^{(j)}_{\lambda}(\mathbf{s}) \,\textup{d}\mathbf{s} \,\textup{d}\mathbf{x}. \]
Without loss of generality we will estimate the piece $\mathcal{P}^{(1)}_{\lambda}(f)$. Part (a) of Lemma~\ref{lm:3combgowers} and \eqref{eq:gowersscaling} give
\[ \big\| \| \tau^{(j)}_{\lambda}(\mathbf{s}) \|_{\textup{U}^2_{s_1}(\mathbb{R}^{d_1})} \big\|_{\textup{L}^{\infty}_{s_2,\ldots,s_n}(\mathbb{R}^{D-d_1})}
\lesssim_D \lambda^{-d_{1}/4} \varepsilon^{1/4} \lambda^{-(d_2+\cdots+d_n)} = \varepsilon^{1/4} \lambda^{-D+3d_1/4}. \]
Then we observe that $(\mathcal{F}f)(\mathbf{x},\mathbf{s})$ can, for fixed $x_2,\ldots,x_n$, $s_2,\ldots,s_n$, be written in the form $f_1(x_1) f_2(x_1+s_1)$ from Lemma~\ref{lm:2combcs}, so we obtain
\[ \Big| \int_{(\mathbb{R}^{d_1})^2} (\mathcal{F}f)(\mathbf{x},\mathbf{s}) \tau^{(j)}_{\lambda}(\mathbf{s}) \,\textup{d}s_1 \,\textup{d}x_1 \Big|
\leq N^{d_1} \lambda^{d_{1}/4} \varepsilon^{1/4} \lambda^{-D+3d_{1}/4} = \varepsilon^{1/4} N^{d_1} \lambda^{-D+d_1}. \]
Integrating in $x_j\in[0,N]^{d_j}$ and $s_j\in[-3\lambda,3\lambda]^{d_j}$, $j=1,2,\ldots,n$ we finally get
\[ \big|\mathcal{P}^{(1)}_{\lambda}(f)\big| \lesssim_{D} \varepsilon^{1/4} N^{d_1} \lambda^{-D+d_1} N^{D-d_1} \lambda^{D-d_1} = \varepsilon^{1/4} N^D, \]
which completes the proof.
\end{proof}


\section{Analytical results}
\label{sec:analysis}

The main ingredient in the proof of Proposition~\ref{prop:3multiscale} is an estimate for multilinear singular integral forms. We formulate it as a separate theorem.

\begin{theorem}\label{thm:singular}\ \
\begin{itemize}
\item[(a)]
Suppose that $K\colon\mathbb{R}^{d_1}\times\cdots\times\mathbb{R}^{d_n}\to\mathbb{C}$ is a bounded compactly supported function and that its Fourier transform satisfies the standard symbol estimates (cf.\@ \cite{S93:habook}),
    \begin{equation}\label{eq:symbolest}
    |\widehat{K}(\xi)| \leq C_{\kappa} \|\xi\|_{\ell^2}^{-|\kappa|}
    \end{equation}
    for any multi-index $\kappa$. Then we have the inequality
\begin{equation}\label{est:thm10a}
 \Big| \int_{(\mathbb{R}^{D})^2} K(\mathbf{s}) \prod_{\mathbf{k}\in\{0,1\}^n} F_{\mathbf{k}}(x_1 + k_1 s_1,\,\dots,\,x_n + k_n s_n) \,\textup{d}\mathbf{s} \,\textup{d}\mathbf{x} \Big| \lesssim \prod_{\mathbf{k}\in\{0,1\}^n} \|F_{\mathbf{k}}\|_{\textup{L}^{2^n}}
\end{equation}
with the implicit constant depending only on $(C_{\kappa})_\kappa$ and the dimensions $d_i$.
\item[(b)]
Suppose that $K\colon\mathbb{R}^{d_1}\times\mathbb{R}^{d_1}\times\mathbb{R}^{d_2}\times\cdots\times\mathbb{R}^{d_n}\to\mathbb{C}$ is a bounded compactly supported function such that its Fourier transform satisfies the standard symbol estimates \eqref{eq:symbolest} for any multi-index $\kappa$. Then we have the inequality
\begin{align*}\nonumber
\bigg| \int_{\mathbb{R}^{3D+d_1}}
& K(s_1,s'_1,s_2,\ldots,s_n)
\Big( \prod_{\substack{\widetilde{\mathbf{k}}=(k_2,\ldots,k_n,l_1,l_2,l_3,l_4)\in\{0,1\}^{n+3}\\ l_1+l_2+l_3+l_4=1}} \\
& F_{\widetilde{\mathbf{k}}}(x_1 + l_1 s_1 + l_2 s'_1,\,x_2 + k_2 s_2,\,\dots,\,x_n + k_n s_n,\,y_1 + l_3 s_1 + l_4 s'_1,\, y_2,\dots, y_n) \Big) \\
& \,\textup{d}s_1 \,\textup{d}s'_1 \,\textup{d}s_2 \cdots \,\textup{d}s_n \,\textup{d}\mathbf{x} \,\textup{d}\mathbf{y} \bigg|
\lesssim \prod_{\substack{\widetilde{\mathbf{k}}=(k_2,\ldots,k_n,l_1,l_2,l_3,l_4)\in\{0,1\}^{n+3}\\ l_1+l_2+l_3+l_4=1}} \|F_{\widetilde{\mathbf{k}}}\|_{\textup{L}^{2^{n+1}}},
\end{align*}
with the implicit constant depending only on $(C_{\kappa})_\kappa$ and the dimensions $d_i$.
\end{itemize}
\end{theorem}

Note that the implicit constants in both inequalities claimed by Theorem~\ref{thm:singular} depend only on the implicit constants from the symbol estimates \eqref{eq:symbolest} and the dimensions, the latter being regarded as fixed throughout the text. Also observe that part (b) of the theorem specialized to $n=1$ coincides with Theorem~1.3 from \cite{DKR17:corner}, the main analytic result of that paper.

Once Theorem~\ref{thm:singular} is established, it is easy to complete the proof of Proposition~\ref{prop:3multiscale}. Let us elaborate on that argument and postpone the proof of the theorem to the second half of this section.

\begin{proof}[Proof of Proposition~\ref{prop:3multiscale}]
For the proof of part (a), note that the cutoff function appearing in $\mathcal{E}^{p,\varepsilon}_{\lambda_m}(f)$ is
\[ \omega^{p,\varepsilon}_{\lambda_m \mathbf{a}}(\mathbf{s}) - b(p,\varepsilon) \omega^{p,1}_{\lambda_m \mathbf{a}}(\mathbf{s})
= \prod_{j=1}^{n} \omega^{d_j,p,\varepsilon}_{\lambda_m a_j}(s_j) - \prod_{j=1}^{n} c(d_j,p,\varepsilon) \omega^{d_j,p,1}_{\lambda_m a_j}(s_j). \]
Recalling the introduction of $k^{d,p,\varepsilon}$ in \eqref{eq:defofk}, we can rewrite it as
\begin{equation}\label{eq:ekernel}
\sum_{j=1}^{n} \Big(\prod_{i=1}^{j-1} \omega^{d_i,p,\varepsilon}_{\lambda_m a_i}(s_i)\Big) k^{d_j,p,\varepsilon}_{\lambda_m a_j}(s_j) \Big(\prod_{i=j+1}^{n} c(d_i,p,\varepsilon) \omega^{d_i,p,1}_{\lambda_m a_i}(s_i)\Big).
\end{equation}
Note that each summand in \eqref{eq:ekernel} is of the form
\begin{equation}\label{eq:ekernel2}
\varphi^{(1)}_{\lambda_m}(s_1) \varphi^{(2)}_{\lambda_m}(s_2) \dots \varphi^{(n)}_{\lambda_m}(s_n),
\end{equation}
where $\varphi^{(j)}$, $j=1,2,\ldots,n$, are $\textup{C}^1$ functions and one of them has integral equal to $0$, while the others are nonnegative.
Take arbitrary signs $\alpha_m\in\{-1,1\}$, $m=1,2,\ldots,M$. By a standard computation (see \cite{CMP15:roth}) the kernels
\[ K(\mathbf{s}) := \sum_{m=1}^{M} \alpha_m \varphi^{(1)}_{\lambda_m}(s_1) \varphi^{(2)}_{\lambda_m}(s_2) \dots \varphi^{(n)}_{\lambda_m}(s_n) \]
satisfy the conditions from Theorem~\ref{thm:singular}, with constants $C_\kappa$ independent of the numbers $M$, $\lambda_1,\ldots,\lambda_M$ and signs $\alpha_1,\ldots,\alpha_M$, but we allow dependencies on the dimensions (i.e.\@ $D$), on $\varepsilon$, on the numbers $a_1,\ldots,a_n$, and on the exponent $p$.
Applying part (a) of Theorem~\ref{thm:singular} to those kernels and $F_\mathbf{k}=f$, we obtain
\[ \Big| \sum_{m=1}^{M} \alpha_m \mathcal{E}^{p,\varepsilon}_{\lambda_m}(f) \Big| \lesssim_{D,\varepsilon} \|f\|_{\textup{L}^{2^n}}^{2^n} \leq N^D. \]
It remains to choose the signs $\alpha_m$ appropriately, so that the left hand side becomes\linebreak $\sum_{m=1}^{M} |\mathcal{E}^{p,\varepsilon}_{\lambda_m}(f)|$.

In the proof of part (b) we begin with the same splitting \eqref{eq:ekernel} into summands of the form \eqref{eq:ekernel2}.
Since the notation has become symmetric in $j$, without loss of generality we can suppose $\int_{\mathbb{R}^{d_1}} \varphi^{(1)}=0$, i.e.\@ the cancellation comes from the variable $s_1$.
Gathering inside parentheses all factors containing that variable, the corresponding part of $\widetilde{\mathcal{E}}^{p,\varepsilon}_{\lambda_m}(f)$ can be rewritten as
{\allowdisplaybreaks\begin{align*}
\int_{\mathbb{R}^{3D-d_1}}
& \bigg( \int_{\mathbb{R}^{d_1}} \Big( \prod_{(k_2,\ldots,k_n)\in\{0,1\}^{n-1}} f(x_1 + s_1,\,x_2 + k_2 s_2,\,\dots,\,x_n + k_n s_n,\,y_1, \dots, y_n) \Big) \\
& \qquad\quad f(x_1, \dots, x_n,\,y_1 + s_1, \dots, y_n) \varphi^{(1)}_{\lambda_m}(s_1) \,\textup{d}s_1 \bigg) \\
& \Big( \prod_{(k_2,\ldots,k_n)\in\{0,1\}^{n-1}} f(x_1,\,x_2 + k_2 s_2,\,\dots,\,x_n + k_n s_n,\,y_1, \dots, y_n) \Big) \\
& \qquad\quad f(x_1, \dots, x_n,\,y_1, y_2 + s_2, \dots, y_n) \cdots f(x_1, \dots, x_n,\,y_1, y_2,\dots, y_n + s_n) \\
& \varphi^{(2)}_{\lambda_m}(s_2) \dots \varphi^{(n)}_{\lambda_m}(s_n) \,\textup{d}s_2 \cdots \,\textup{d}s_n \,\textup{d}\mathbf{x} \,\textup{d}\mathbf{y}.
\end{align*}}
By the Cauchy--Schwarz inequality its square is bounded with
\begin{equation}\label{eq:propaux1}
\mathcal{A}_{\lambda_m}(f) \mathcal{B}_{\lambda_m}(f),
\end{equation}
where
\begin{align*}
\mathcal{A}_{\lambda}(f) := \int_{\mathbb{R}^{3D+d_1}}
& \Big( \prod_{\substack{\widetilde{\mathbf{k}}=(k_2,\ldots,k_n,l_1,l_2,l_3,l_4)\in\{0,1\}^{n+3}\\ l_1+l_2+l_3+l_4=1}} \\
& f_{\widetilde{\mathbf{k}}}(x_1 + l_1 s_1 + l_2 s'_1,\,x_2 + k_2 s_2,\,\dots,\,x_n + k_n s_n,\,y_1 + l_3 s_1 + l_4 s'_1,\, y_2,\dots, y_n) \Big) \\
& \varphi^{(1)}_{\lambda}(s_1) \varphi^{(1)}_{\lambda}(s'_1) \varphi^{(2)}_{\lambda}(s_2) \dots \varphi^{(n)}_{\lambda}(s_n) \,\textup{d}s_1 \,\textup{d}s'_1 \,\textup{d}s_2 \cdots \,\textup{d}s_n \,\textup{d}\mathbf{x} \,\textup{d}\mathbf{y}
\end{align*}
and
\begin{align*}
\mathcal{B}_{\lambda}(f) := \int_{\mathbb{R}^{3D-d_1}}
& \Big( \prod_{(k_2,\ldots,k_n)\in\{0,1\}^{n-1}} \mathbbm{1}_{[0,N]^{2D}}(x_1,\,x_2 + k_2 s_2,\,\dots,\,x_n + k_n s_n,\,y_1, \dots, y_n) \Big) \\
& \mathbbm{1}_{[0,N]^{2D}}(x_1, \dots, x_n,\,y_1, y_2 + s_2, \dots, y_n) \\
& \cdots \mathbbm{1}_{[0,N]^{2D}}(x_1, \dots, x_n,\,y_1, y_2,\dots, y_n + s_n) \\
& \varphi^{(2)}_{\lambda}(s_2) \dots \varphi^{(n)}_{\lambda}(s_n) \,\textup{d}s_2 \cdots \,\textup{d}s_n \,\textup{d}\mathbf{x} \,\textup{d}\mathbf{y}.
\end{align*}
The functions $f_{\widetilde{\mathbf{k}}}$ appearing in the definition of $\mathcal{A}_{\lambda}(f)$ are all supported in $[0,N]^{2D}$ and taking values in $[0,1]$; some of them are equal to $f$, while the others are artificially inserted into the expression as $\mathbbm{1}_{[0,2N]^{2D}}$.
Clearly,
\begin{equation}\label{eq:propaux2}
\mathcal{B}_{\lambda_m}(f)
\leq \Big( \prod_{j=2}^{n} \int_{\mathbb{R}^{d_j}} \varphi^{(j)}_{\lambda_m} \Big) \Big( \prod_{j=1}^{n} \int_{\mathbb{R}^{d_j}} \mathbbm{1}_{[0,N]^{d_j}} \Big)^2
\lesssim_{D,\varepsilon} N^{2D}.
\end{equation}
On the other hand, part (b) of Theorem~\ref{thm:singular} applied with the kernel
\[ K(\mathbf{s}) := \sum_{m=1}^{M} \varphi^{(1)}_{\lambda_m}(s_1) \varphi^{(1)}_{\lambda_m}(s'_1) \varphi^{(2)}_{\lambda_m}(s_2) \dots \varphi^{(n)}_{\lambda_m}(s_n) \]
and the functions $F_{\widetilde{\mathbf{k}}}=f_{\widetilde{\mathbf{k}}}$ yields
\begin{equation}\label{eq:propaux3}
\Big| \sum_{m=1}^{M} \mathcal{A}_{\lambda_m}(f) \Big|
\lesssim_{D,\varepsilon} N^{2D}.
\end{equation}
Summing the products \eqref{eq:propaux1} in $m=1,2,\ldots,M$, from \eqref{eq:propaux2} and \eqref{eq:propaux3} we finally conclude
\[ \sum_{m=1}^{M} |\widetilde{\mathcal{E}}^{p,\varepsilon}_{\lambda_m}(f)|^2 \lesssim_{D,\varepsilon} N^{4D}, \]
as claimed.
\end{proof}

We finalize the paper with the proof of the remaining analytical result.
As we will soon see, the case when all dimensions $d_i$ are equal will be an easy consequence of the main result from the paper by Thiele and one of the present authors \cite{DT18}.
We will spend just a slight additional effort to reduce the case of possibly different dimensions $d_i$ to the very same result.
An alternative to this addition could be considering appropriate lower-dimensional sections of the sets and functions appearing throughout the paper.

\begin{proof}[Proof of Theorem~\ref{thm:singular}]
Proof of (a).
First we consider the case of equal dimensions, i.e.\@ $d_1 = d_2 =\cdots = d_n$, and we write them simply as $d$.
Relabeling $x_i$ to $x_i^0$ in the left hand-side of \eqref{est:thm10a} and changing variables $x_i^0 + s_i = x_i^1$
for $i=1,\ldots,n$, we see that we need to show
\begin{align*}
\Big| \int_{\mathbb{R}^{2dn}} \prod_{\mathbf{k}\in \{0,1\}^n} F_{\mathbf{k}} (\Pi_\mathbf{k} x) K(\Pi x) \,\textup{d}x \Big |
\lesssim_{(C_{\kappa})_\kappa} \prod_{\mathbf{k}\in\{0,1\}^n} \|F_{\mathbf{k}}\|_{\textup{L}^{2^n}(\mathbb{R}^{nd})},
\end{align*}
where $x=(x_1^0,\ldots, x_n^0,x_1^1,\ldots x_n^1)\in (\mathbb{R}^d)^{2n}$ and $\Pi_\mathbf{k},\Pi : (\mathbb{R}^{d})^{2n} \rightarrow (\mathbb{R}^{d})^n$ are linear operators given by
\begin{align*}
\Pi_\mathbf{k}x := (x_1^{k_1},\ldots, x_n^{k_n}),\quad \Pi x := (x_1^1-x_1^0,\ldots, x_n^1-x_n^0).
\end{align*}
This estimate can be recognized as one of the singular Brascamp--Lieb inequalities from the main theorem of \cite{DT18}, which establishes the claim in the case of equal dimensions.

The general case of different dimensions in \eqref{est:thm10a} will be deduced from the case of equal dimensions as follows. By approximating $K$ with smooth compactly supported functions in $\textup{L}^1$ and applying H\"older's inequality, we may assume  that $K$ is a smooth compactly supported function on $\mathbb{R}^D \cong \mathbb{R}^{d_1}\times \cdots \times \mathbb{R}^{d_n}$, satisfying the symbol estimates \eqref{eq:symbolest}. We set
\[d:=\max_{1\leq i \leq n} d_i+1.\]
(We have added $1$ for technical reasons, so that $d-d_i>0$ for all $1\leq i \leq n$.)
First we define a new function $\widetilde{K}$ on $\mathbb{R}^{nd}$, whose integration in a certain direction gives $K$.
To achieve this take a smooth compactly supported function $\varphi$ on $\mathbb{R}^{nd-D}\cong \mathbb{R}^{d-d_1}\times \cdots \times \mathbb{R}^{d-d_n}$ satisfying $\varphi(\mathbf{0})=1$. Then we define $\widetilde{K}$ by setting
\begin{align}\label{eq:defKtilde}
\widehat{\widetilde{K}}(\xi,\widetilde{\xi}) := \widehat{K}(\xi) \varphi \Big ( \frac{\widetilde{\xi}}{\|\xi\|} \Big ).
\end{align}
for $\mathbf{0}\neq \xi\in \mathbb{R}^{D}$, $\widetilde{\xi}\in \mathbb{R}^{nd-D}$.
Observe that we have
\begin{align*}
\widehat{\widetilde{K}}(\xi,0) = \widehat{K}(\xi)
\end{align*}
or, passing to the spatial side,
 \begin{align*}
K(\mathbf{s}) = \int_{\mathbb{R}^{nd-D}} \widetilde{K}(\mathbf{s},\mathbf{\widetilde{s}}) \,\textup{d}\mathbf{\widetilde{s}}.
\end{align*}
Moreover, the function $\widehat{\widetilde{K}}$ satisfies the symbol estimates
    \begin{equation*}
    \big|\widehat{\widetilde{K}}(\xi,\widetilde{\xi})\big| \leq \widetilde{C}_{\kappa} \|(\xi,\widetilde{\xi})\|_{\ell^2}^{-|\kappa|}
    \end{equation*}
for all multi-indices $\kappa$ and all $(\xi,\widetilde{\xi})\neq 0$, with $\widetilde{C}_{\kappa}$ depending only on $C_\kappa$.

Assuming the estimate \eqref{est:thm10a} in the case of equal dimensions,
let us plug in the kernel $\widetilde{K}$ and the functions $\widetilde{F}_{\mathbf{k}}: (\mathbb{R}^d)^n\rightarrow \mathbb{C}$ defined by
\begin{align*}
\widetilde{F}_{\mathbf{k}}(\mathbf{z},\mathbf{\widetilde{z}}) := F_{\mathbf{k}}(\mathbf{z}) (\textup{D}_{\lambda}^{2^n}\varphi_1\otimes \cdots \otimes \textup{D}_{\lambda}^{2^n}\varphi_n)(\mathbf{\widetilde{z}}),
\end{align*}
where $\mathbf{z}\in \mathbb{R}^{D},\,\mathbf{\widetilde{z}}\in \mathbb{R}^{nd-D}$, $F_{\mathbf{k}}:\mathbb{R}^{D} \rightarrow \mathbb{C}$, $\varphi_i$ is a Schwartz function on $\mathbb{R}^{d-d_i}$ and $\textup{D}_{\lambda}^{2^n} \varphi_i$ is defined by \eqref{eq:funcdilatep}. Then
the form in question becomes
\begin{align}\nonumber
\int_{\mathbb{R}^{2nd}} \widetilde{K}(\mathbf{s},\mathbf{\widetilde{s}})\, \Big (&\, \prod_{\mathbf{k}\in\{0,1\}^n} F_{\mathbf{k}}(x_1 + k_1 s_1,\,\dots,\,x_n + k_n s_n) \\
& \textup{D}_{\lambda}^{2^n}\varphi_1(\widetilde{x}_1 + k_1 \widetilde{s}_1)\cdots \textup{D}_{\lambda}^{2^n}\varphi_n(\widetilde{x}_n + k_n \widetilde{s}_n) \Big )
\,\textup{d}\mathbf{s}\,\textup{d}\mathbf{\widetilde{s}} \,\textup{d}\mathbf{x} \,\textup{d}\mathbf{\widetilde{x}}, \label{eq:lhs}
\end{align}
where $\mathbf{\widetilde{s}} = (\widetilde{s}_1,\ldots , \widetilde{s}_n),\, \mathbf{\widetilde{x}} = (\widetilde{x}_1,\ldots , \widetilde{x}_n)\in \mathbb{R}^{nd-D}$.
From the case of equal dimensions we know that the last display is bounded by
\begin{align}
\prod_{\mathbf{k}\in \{0,1\}^n} \|\widetilde{F}_{\mathbf{k}}\|_{\textup{L}^{2^n}(\mathbb{R}^{nd})} = \|\varphi_1\otimes \cdots \otimes \varphi_n\|^{2^n}_{\textup{L}^{2^n}(\mathbb{R}^{nd-D})}\prod_{\mathbf{k}\in \{0,1\}^n} \|F_{\mathbf{k}}\|_{\textup{L}^{2^n}(\mathbb{R}^{D})} \label{eq:boundconstant}
\end{align}
times a constant depending only on $(C_\kappa)_\kappa$.
On the other hand, \eqref{eq:lhs} equals
\begin{align}\nonumber
\int_{\mathbb{R}^{2nd}} & \widetilde{K}(\mathbf{s},\mathbf{\widetilde{s}}) \Big( \prod_{i=1}^n \lambda^{-(d-d_i)}\varphi_i(\lambda^{-1}\widetilde{x}_i)^{2^{n-1}}\varphi_i(\lambda^{-1}(\widetilde{x}_i + \widetilde{s}_i))^{2^{n-1}} \Big ) \\
\label{eq:lhs2}
& \prod_{\mathbf{k}\in\{0,1\}^n} F_{\mathbf{k}}(x_1 + k_1 s_1,\,\dots,\,x_n + k_n s_n) \,\textup{d}\mathbf{s}\,\textup{d}\mathbf{\widetilde{s}} \,\textup{d}\mathbf{x} \,\textup{d}\mathbf{\widetilde{x}} .
\end{align}
Integrating in $\mathbf{\widetilde{x}}$ we obtain
\begin{align}\label{eq:analytic1}
\int_{\mathbb{R}^{nd+D}} \widetilde{K}(\mathbf{s},\mathbf{\widetilde{s}}) \Phi(\lambda^{-1}\mathbf{\widetilde{s}}) \prod_{\mathbf{k}\in\{0,1\}^n} F_{\mathbf{k}}(x_1 + k_1 s_1,\,\dots,\,x_n + k_n s_n)
 \,\textup{d}\mathbf{s}\,\textup{d}\mathbf{\widetilde{s}} \,\textup{d}\mathbf{x},
\end{align}
where we have set
\[\Phi:=\big(  \varphi_1^{2^{n-1}}\ast \widetilde{\varphi}_1^{\,2^{n-1}} \big )\otimes \cdots \otimes \big( \varphi_n^{\, 2^{n-1}}\ast \widetilde{\varphi}_n^{2^{n-1}}\big ) \]
and $\widetilde{\varphi}_i(s):=\varphi_i(-s).$
Integrating in $\mathbf{\widetilde{s}}$ and taking the limit as $\lambda\rightarrow \infty$, \eqref{eq:analytic1} becomes, up to a constant,
\begin{align}\label{eq:analytic2}
\int_{(\mathbb{R}^D)^2}{K}(\mathbf{s}) \prod_{\mathbf{k}\in\{0,1\}^n} F_{\mathbf{k}}(x_1 + k_1 s_1,\,\dots,\,x_n + k_n s_n)
 \,\textup{d}\mathbf{s} \,\textup{d}\mathbf{x}
\end{align}
and we know that it is bounded by \eqref{eq:boundconstant}, as desired.

To justify passage to the limit we observe that the difference of \eqref{eq:analytic1} and \eqref{eq:analytic2} equals
\begin{align*}
\int_{(\mathbb{R}^D)^2} \Big ( \int_{\mathbb{R}^{nd-D}} \widetilde{K}(\mathbf{s},\mathbf{\widetilde{s}}) \big(\Phi(\lambda^{-1}\mathbf{\widetilde{s}}) - 1\big) \,\textup{d}\mathbf{\widetilde{s}}\Big ) \prod_{\mathbf{k}\in\{0,1\}^n} F_{\mathbf{k}}(x_1 + k_1 s_1,\,\dots,\,x_n + k_n s_n)
 \, \textup{d}\mathbf{s}\, \,\textup{d}\mathbf{x},
\end{align*}
which tends to zero as $\lambda\rightarrow \infty$.
Indeed, this follows by applying
H\"older's inequality in $\mathbf{x}$, which bounds the last display by
\begin{align*}
\Big( \int_{\mathbb{R}^D} \Big | \int_{\mathbb{R}^{nd-D}} \widetilde{K}(\mathbf{s},\mathbf{\widetilde{s}}) \big(\Phi(\lambda^{-1}\mathbf{\widetilde{s}}) - 1\big) \,\textup{d}\mathbf{\widetilde{s}} \, \Big | \,\textup{d}\mathbf{s} \Big ) \prod_{\mathbf{k}\in \{0,1\}^n} \|{F}_{\mathbf{k}}\|_{\textup{L}^{2^n}(\mathbb{R}^{D})} .
\end{align*}
We note that the expression in the bracket
tends to zero as $\lambda \rightarrow \infty$, as desired.

Proof of (b).
Note that it suffices to show the bound
\begin{align}\nonumber
\bigg| \int_{\mathbb{R}^{2D+2d_1}}
& K(\mathbf{s},s'_1)
\Big( \prod_{\substack{\widetilde{\mathbf{k}}=(k_2,\ldots,k_n,l_1,l_2,l_3,l_4)\in\{0,1\}^{n+3}\\ l_1+l_2+l_3+l_4=1}} \\ \nonumber
& F_{\widetilde{\mathbf{k}}}(x_1 + l_1 s_1 + l_2 s'_1,\,x_2 + k_2 s_2,\,\dots,\,x_n + k_n s_n,\,y_1 + l_3 s_1 + l_4 s'_1) \Big) \\ \label{est:thm10b}
& \,\textup{d}\mathbf{s} \,\textup{d}s'_1 \,\textup{d}\mathbf{x} \,\textup{d}y_1\bigg|
\lesssim_{(C_{\kappa})_\kappa} \prod_{\substack{\widetilde{\mathbf{k}}=(k_2,\ldots,k_n,l_1,l_2,l_3,l_4)\in\{0,1\}^{n+3}\\ l_1+l_2+l_3+l_4=1}} \|F_{\widetilde{\mathbf{k}}}\|_{\textup{L}^{2^{n+1}}(\mathbb{R}^{(n+1)d})}
\end{align}
for functions $F_{\widetilde{\mathbf{k}}}:\mathbb{R}^{D+d_1}\rightarrow \mathbb{C}$, where $\mathbf{s}=(s_1,\ldots, s_n)\in \mathbb{R}^{D}$. Indeed,
part (b) of Theorem \ref{thm:singular} then follows by Fubini, applying the estimate \eqref{est:thm10b} and H\"older's inequality in $y_2,\ldots , y_n$.

To show \eqref{est:thm10b} we again first consider the case of equal dimensions, $d_1= \cdots =d_n$. Once again, we write them simply as $d$.
Changing variables $x_1+y_1+s_1 = x_{n+1}^0$, $x_1+y_1+s_1'=x_{n+1}^1$, we obtain
{\allowdisplaybreaks\begin{align*}
\int_{\mathbb{R}^{2d(n+1)}}
& K(x_{n+1}^0-x_1-y_1,s_2,\ldots,s_n,x_{n+1}^1-x_1-y_1) \Big(\,\prod_{k=(k_2,\ldots,k_n)\in \{0,1\}^{n-1}} \\
 & F_{(k,e_1)}(x_{n+1}^0-y_1, x_2 + k_2 s_2,\,\dots,\,x_n + k_n s_n,y_1)\\
& F_{(k,e_2)}(x_{n+1}^1-y_1, x_2 + k_2 s_2,\,\dots,\,x_n + k_n s_n,y_1)\\
& F_{(k,e_3)}(x_1, x_2 + k_2 s_2,\,\dots,\,x_n + k_n s_n,\,x_{n+1}^0-x_1)\\
& F_{(k,e_4)}(x_1, x_2 + k_2 s_2,\,\dots,\,x_n + k_n s_n,\,x_{n+1}^1-x_1) \Big) \,\textup{d}x_{n+1}^0 \,\textup{d}x_{n+1}^1\,\textup{d}s_2 \cdots \,\textup{d}s_n \,\textup{d}\mathbf{x} \,\textup{d}y_1,
\end{align*}}
where $e_i$ are standard unit vectors in $\mathbb{R}^4$.
Shearing the functions $F_{(k,e_i)}$ we see that it suffices to show estimate an estimate for the form
{\allowdisplaybreaks\begin{align*}
\int_{\mathbb{R}^{2d(n+1)}}
&K(x_{n+1}^0-x_1-y_1,s_2,\ldots,s_n,x_{n+1}^1-x_1-y_1) \Big( \, \prod_{k=(k_2,\ldots,k_n)\in \{0,1\}^{n-1}} \\
& F_{(k,e_1)}(x_{n+1}^0, x_2 + k_2 s_2,\,\dots,\,x_n + k_n s_n,y_1)\\
& F_{(k,e_2)}(x_{n+1}^1, x_2 + k_2 s_2,\,\dots,\,x_n + k_n s_n,\,y_1)\\
& F_{(k,e_3)}(x_1, x_2 + k_2 s_2,\,\dots,\,x_n + k_n s_n,\,x_{n+1}^0)\\
& F_{(k,e_4)}(x_1, x_2 + k_2 s_2,\,\dots,\,x_n + k_n s_n,\,x_{n+1}^1) \Big) \,\textup{d}x_{n+1}^0 \,\textup{d}x_{n+1}^1 \,\textup{d}s_2 \cdots \,\textup{d}s_n \,\textup{d}\mathbf{x} \,\textup{d}y_1.
\end{align*}}
We relabel $y_1$ into $x_1^1$, $x_i$ into $x_i^0$ for $1\leq i \leq n$, and change variables $x_i^0 + s_i = x_i^1$ for $2\leq i \leq n$.
Then we see that it suffices to show the estimate
\begin{align*}
\Big| \int_{\mathbb{R}^{2d(n+1)}} \prod_{\mathbf{k}\in \{0,1\}^{n+1}} F_{\mathbf{k}} (\Pi_\mathbf{k} x) K(\Pi x) \,\textup{d}x \Big |
\lesssim_{(C_{\kappa})_\kappa} \prod_{\mathbf{k}\in\{0,1\}^n} \|F_{\mathbf{k}}\|_{\textup{L}^{2^{n+1}}(\mathbb{R}^{(n+1)d})}
\end{align*}
for
$x=(x_1^0,\ldots, x_{n+1}^0,x_1^1,\ldots x_{n+1}^1)\in (\mathbb{R}^d)^{n+1}$ and linear operators $\Pi_\mathbf{k},\Pi : (\mathbb{R}^{d})^{2(n+1)} \rightarrow (\mathbb{R}^{d})^{n+1}$ given by
\begin{align*}
\Pi_\mathbf{k}x := (x_1^{k_1},\ldots, x_{n+1}^{k_{n+1}}),\quad \Pi x := (x_{n+1}^0-x_1^1-x_1^0, x_2^1-x_2^0,\ldots, x_{n}^1-x_{n}^0, x_{n+1}^1-x_1^1-x_1^0).
\end{align*}
This estimate again follows from the main result in \cite{DT18}.

To finish the proof of \eqref{est:thm10b} it remains to deduce the case of different dimensions from the case of equal dimensions. This follows similarly as in part (a) and we only sketch the necessary modifications.
Let $d$ be defined as in the proof of part (a) of this theorem.
Assuming the estimate \eqref{est:thm10b} in the case of equal dimensions,
let us plug in the kernel $\widetilde{K}$ defined on $(\mathbb{R}^d)^{n+1}$ as in \eqref{eq:defKtilde} and the functions $F_{\mathbf{\widetilde{k}}}\colon (\mathbb{R}^d)^{n+1}\rightarrow \mathbb{C}$ given by
\begin{align*}
\widetilde{F}_{\mathbf{\widetilde{k}}}(\mathbf{z},\mathbf{\widetilde{z}}) := F_{\mathbf{\widetilde{k}}}(\mathbf{z}) (\textup{D}_{\lambda}^{2^{n+1}}\varphi_1\otimes \cdots \otimes \textup{D}_{\lambda}^{2^{n+1}}\varphi_{n} \otimes \textup{D}_{\lambda}^{2^{n+1}}\varphi_{1})(\mathbf{\widetilde{z}}),
\end{align*}
where $\mathbf{z}\in \mathbb{R}^{D+d_1},\,\mathbf{\widetilde{z}}\in \mathbb{R}^{(n+1)d-D-d_1}$, $F_{\mathbf{\widetilde{k}}}:\mathbb{R}^{D+d_1}\rightarrow \mathbb{C}$, $\varphi_i$ is a Schwartz function on $\mathbb{R}^{d-d_i}$ for $1\leq i \leq n$, and $\textup{D}_{\lambda}^{2^{n+1}} \varphi_i$ was defined in \eqref{eq:funcdilatep}. Then the form in \eqref{est:thm10b} becomes, analogously to the display \eqref{eq:lhs2} in part (a),
{\allowdisplaybreaks\begin{align}
\int_{\mathbb{R}^{2d(n+1)}} \nonumber
& \widetilde{K}(\mathbf{s}, s'_1)\,
\lambda^{-(d-d_1)}\varphi_1(\lambda^{-1}\widetilde{x}_1)^{2^{n}}\varphi_1(\lambda^{-1}(\widetilde{x}_1 + \widetilde{s}_1))^{2^{n-1}} \varphi_1(\lambda^{-1}(\widetilde{x}_1 + \widetilde{s}'_1))^{2^{n-1}} \\ \nonumber
& \Big( \prod_{i=2}^n \lambda^{-(d-d_i)}\varphi_i(\lambda^{-1}\widetilde{x}_i)^{2^{n}}\varphi_i(\lambda^{-1}(\widetilde{x}_i + \widetilde{s}_i))^{2^{n}} \Big ) \\ \nonumber
&\lambda^{-(d-d_1)}\varphi_{1}(\lambda^{-1}\widetilde{y}_1)^{2^{n}}\varphi_{1}(\lambda^{-1}(\widetilde{y}_1+ \widetilde{s}_{1}))^{2^{n-1}} \varphi_{1}(\lambda^{-1}(\widetilde{y}_1 + \widetilde{s}'_{1}))^{2^{n-1}} \\ \nonumber
&\Big (\, \prod_{\substack{\widetilde{\mathbf{k}}=(k_2,\ldots,k_n,l_1,l_2,l_3,l_4)\in\{0,1\}^{n+3}\\ l_1+l_2+l_3+l_4=1}} \\ \nonumber
&F_{\widetilde{\mathbf{k}}}(x_1 + l_1 s_1 + l_2 s'_1,\,x_2 + k_2 s_2,\,\dots,\,x_n + k_n s_n,\,y_1 + l_3 s_1 + l_4 s'_1)\, \Big ) \\ \label{form:partb}
& \,\textup{d}\mathbf{s}\,\textup{d}s'_1 \,\textup{d}\mathbf{\widetilde{s}} \,\textup{d}\widetilde{s}'_1 \,\textup{d}\mathbf{x} \,\textup{d}\mathbf{\widetilde{x}} \,\textup{d}y_1 \,\textup{d}\widetilde{y_1},
\end{align}}
where $\mathbf{\widetilde{s}} = (\widetilde{s}_1,\ldots, \widetilde{s}_n)\in \mathbb{R}^{nd-D}, \,\widetilde{s}'_1\in \mathbb{R}^{d-d_1},\, \widetilde{y}_1\in \mathbb{R}^{d-d_1}$.
From the case of equal dimensions we know that it is bounded by a constant times
\begin{align*}
\prod_{\mathbf{\widetilde{k}}} \|\widetilde{F}_{\mathbf{\widetilde{k}}}\|_{\textup{L}^{2^{n+1}}(\mathbb{R}^{(n+1)d})}
= \|\varphi_1\otimes \cdots \otimes \varphi_n \otimes \varphi_1 \|^{2^{n+1}}_{\textup{L}^{2^{n+1}}(\mathbb{R}^{(n+1)d-D-d_1})}
\prod_{\mathbf{\widetilde{k}}} \|F_{\mathbf{k}}\|_{\textup{L}^{2^{n+1}}(\mathbb{R}^{D+d_1})}.
\end{align*}

On the other hand, integrating in $\mathbf{\widetilde{x}}$ and $\widetilde{y}_1$ gives that the form \eqref{form:partb} equals
\begin{align*}
\int_{\mathbb{R}^{2d(n+1)}}
& \widetilde{K}(\mathbf{s}, s'_1)\,
\Phi_1(\lambda^{-1}\widetilde{s_1},\lambda^{-1}\widetilde{s}_1') ^2 \Phi (\lambda^{-1}(\widetilde{s}_2,\ldots,\widetilde{s}_n))
\, \Big(\, \prod_{\substack{\widetilde{\mathbf{k}}=(k_2,\ldots,k_n,l_1,l_2,l_3,l_4)\in\{0,1\}^{n+3}\\ l_1+l_2+l_3+l_4=1}} \\
& \!\!\! F_{\widetilde{\mathbf{k}}}(x_1 + l_1 s_1 + l_2 s'_1,\,x_2 + k_2 s_2,\,\dots,\,x_n + k_n s_n,\,y_1 + l_3 s_1 + l_4 s'_1)\, \Big) \,\textup{d}\mathbf{s}\,\textup{d}s'_1 \,\textup{d}\mathbf{\widetilde{s}} \,\textup{d}\widetilde{s}'_1 \,\textup{d}\mathbf{x} \,\textup{d}y_1,
\end{align*}
where
\[ \Phi_1(\widetilde{s}_1,\widetilde{s}'_1) := \int_{\mathbb{R}^{d-d_1}} \varphi_i(u)^{2^{n}}\varphi_i(u + \widetilde{s}_1)^{2^{n-1}} \varphi_1(u + \widetilde{s}'_1)^{2^{n-1}} \,\textup{d}u \]
and
\[ \Phi:= \big ( \varphi_2^{2^{n}}\ast \widetilde{\varphi}_2^{\, 2^{n}}   \big )\otimes \cdots \otimes \big ( \varphi_n^{2^{n}}\ast \widetilde{\varphi}_n^{\, 2^{n}} \big ), \]
where $\widetilde{\varphi}_i$ is defined as in (a).
Taking the limit as $\lambda \rightarrow \infty$ and integrating in $\mathbf{\widetilde{s}}, \widetilde{s}_1'$, similarly as in the proof of part (a), we recover the form on the left hand-side of \eqref{est:thm10b}.
\end{proof}


\section*{Acknowledgments}
The authors thank Christoph Thiele for inspiring discussions aided by the bilateral DAAD-MZO grant \emph{Multilinear singular integrals and applications}.
The authors are also grateful to the anonymous referee for useful suggestions.
V. K. was supported in part by the Croatian Science Foundation under the project UIP-2017-05-4129 (MUNHANAP).


\end{document}